\documentclass[11pt]{amsart}

\usepackage{amscd}
\usepackage{amsfonts}
\usepackage{amsmath}
\usepackage{amssymb}
\usepackage{amsthm}
\usepackage[UKenglish]{babel}
\usepackage{bbm} 
\usepackage{cases}
\usepackage{chngcntr}
\usepackage{cite}
\usepackage{color}
\usepackage{graphicx}
\usepackage{latexsym}
\usepackage{mathtools}
\usepackage{microtype}
\usepackage{qsymbols}
\usepackage[active]{srcltx}
\usepackage{url}

\usepackage{hyperref}
\usepackage{cleveref}

\usepackage{tikz,tikz-cd}
\usetikzlibrary{matrix, arrows}

\makeatletter
\tikzset{
	edge node/.code={%
		\expandafter\def\expandafter\tikz@tonodes\expandafter{\tikz@tonodes #1}}}
\makeatother
\tikzset{
	subseteq/.style={
		draw=none,
		edge node={node [sloped, allow upside down, auto=false]{$\subseteq$}}},
	Subseteq/.style={
		draw=none,
		every to/.append style={
			edge node={node [sloped, allow upside down, auto=false]{$\subset$}}}}
}

\newtheorem{theorem}{Theorem}[section]
\newtheorem{proposition}[theorem]{Proposition}
\newtheorem{lemma}[theorem]{Lemma}
\newtheorem{corollary}[theorem]{Corollary}

\theoremstyle{definition}
\newtheorem{definition}[theorem]{Definition}

\theoremstyle{remark}
\newtheorem{remark}[theorem]{Remark}

\usepackage{calrsfs}
\usepackage[charter]{mathdesign}
\newcommand{\bbfont}{\mathbb}

\newcommand{\catfont}[1]{\textup{\fontfamily{qag}\selectfont{\footnotesize #1}}}
\newcommand{\catfontscript}[1]{\textup{\fontfamily{qag}\selectfont{\tiny{#1}}}}

\newcommand{\RR}{{\bbfont R}}

\newcommand{\norm}[1]{{\lVert #1 \rVert}}

\newcommand{\Fun}{\mathrm{Fun}\hskip 0.07em}
\newcommand{\co}{\mathrm{co}\hskip 0.07em}
\newcommand{\fact}[1]{{\overline{#1}}}

\newcommand{\catsymbol}{\catfont{Cat}}
\newcommand{\catone}{\catsymbol_1}
\newcommand{\cattwo}{\catsymbol_2}
\newcommand{\catthree}{\catsymbol_3}

\newcommand{\catsymbolscript}{\catfontscript{Cat}}
\newcommand{\catonescript}{{\catsymbolscript}_1}
\newcommand{\cattwoscript}{{\catsymbolscript}_2}
\newcommand{\catthreescript}{\catsymbolscript_3}

\newcommand{\functor}{U}

\newcommand{\SET}{\catfont{Set}}
\newcommand{\VS}{\catfont{VS}}
\newcommand{\VL}{\catfont{VL}}

\newcommand{\SETscript}{\catfontscript{Set}}
\newcommand{\VSscript}{\catfontscript{VS}}
\newcommand{\VLscript}{\catfontscript{VL}}

\newcommand{\BS}{\catfont{BS}}
\newcommand{\BL}{\catfont{BL}}

\newcommand{\BSscript}{\catfontscript{BS}}
\newcommand{\BLscript}{\catfontscript{BL}}

\newcommand{\freeletter}{\textup{F}}
\newcommand{\Free}[3]{\freeletter_{#1}^{#2}\!\left[#3\right]}

\newcommand{\FCATONECATTWO}[1]{\Free{\catonescript}{\cattwoscript}{#1}}
\newcommand{\FCATTWOCATTHREE}[1]{\Free{\cattwoscript}{\catthreescript}{#1}}
\newcommand{\FCATONECATTHREE}[1]{\Free{\catonescript}{\catthreescript}{#1}}

\newcommand{\FSETVS}[1]{\Free{\SETscript}{\VSscript}{#1}}
\newcommand{\FSETVL}[1]{\Free{\SETscript}{\VLscript}{#1}}
\newcommand{\FSETBL}[1]{\Free{\SETscript}{\BLscript}{#1}}
\newcommand{\FSETBS}[1]{\Free{\SETscript}{\BSscript}{#1}}
\newcommand{\FVSVL}[1]{\Free{\VSscript}{\VLscript}{#1}}
\newcommand{\FBSBL}[1]{\Free{\BSscript}{\BLscript}{#1}}

\numberwithin{equation}{section}
\allowdisplaybreaks

\usepackage[nobysame]{amsrefs}
\AtBeginDocument{\def\MR#1{}}

\crefname{theorem}{Theorem}{Theorems}
\crefname{proposition}{Proposition}{Propositions}
\crefname{lemma}{Lemma}{Lemmas}
\crefname{corollary}{Corollary}{Corollaries}
\crefname{definition}{Definition}{Definitions}
\crefname{example}{Example}{Examples}
\crefname{examples}{Examples}{Examples}
\crefname{remark}{Remark}{Remarks}
\crefname{equation}{equation}{equations}

\crefname{section}{Section}{Sections}
\crefname{subsection}{Section}{Sections}
\crefname{subsubsection}{Section}{Sections}

\begin{document}

%
%
%
%
%


\title [Free vector lattices over vector spaces as function lattices]{Free vector lattices over vector spaces as function lattices}

\author[Marcel de Jeu]{Marcel de Jeu}

\address{Mathematical Institute\\
	Leiden University\\
	P.O.\ Box 9512\\
	2300 RA Leiden\\
	the Netherlands\\
	and\\
	Department of Mathematics and Applied Mathematics\\
	University of Pretoria\\
	Corner of Lynnwood Road and Roper Street\\
	Hatfield 0083\\
	Pretoria\\
	South Africa
}
\email{mdejeu@math.leidenuniv.nl}



\begin{abstract}
	We show that a free vector lattice over a real vector space $V$ can be realised canonically as a vector lattice of real-valued positively homogeneous functions on any linear subspace of its dual space that separates the points of $V$. This is used to give intuition for the known fact that the free Banach lattice over a Banach space $E$ can be realised as a Banach lattice of positively homogeneous functions on $E^\ast$. It is also applied to improve the well-known result that free vector lattices over non-empty sets can be realised as vector lattices of real-valued functions. For infinite sets, the underlying spaces for such realisations can be chosen to be smaller than the usual ones.
\end{abstract}

\maketitle


\section{Introduction and context}\label{sec:introduction_and_context}

The first papers in which Banach lattices with universal properties are considered appear to be \cite{fremlin:1974} and \cite{haydon:1977}. Some forty years later, the interest in this was revived by de Pagter and Wickstead in \cite{de_pagter_wickstead:2015} and subsequently by Avil\'{e}s, Rodr\'{\i}guez, and Tradacete in \cite{aviles_rodriguez_tradacete:2018}. This has been followed up by a good number of papers on free and projective Banach lattices; see \cite{aviles_martinez-cervantes_rodriguez_tradacete:2022,aviles_martinez-cervantes_rodriguez-abellan:2020,aviles_martinez-cervantes_rodriguez-abellan:2021,aviles_martinez-cervantes_rodriguez-abellan_rueda-zoca:2022a,aviles_martinez-cervantes_rodriguez-abellan_rueda-zoca:2022b,aviles_plebanek_rodriguez-abellan:2018,aviles_rodriguez-abellan_UNPUBLISHED:2020a,aviles_rodriguez-abellan_UNPUBLISHED:2020b,aviles_rodriguez-abellan:2019_a,aviles_rodriguez-abellan:2019_b,aviles_tradacete_villanueva:2019,dantas_martinez-cervantes_rodriguez-abellan_rueda-zoca:2021,dantas_martinez-cervantes_rodriguez-abellan_rueda_zoca:2022,de_hevia_tradacete_UNPUBLISHED:2022,jardon-sanchez_laustsen_taylor_tradacete_troitsky:2022,oikhberg_taylor_tradacete_troitsky_UNPUBLISHED:2022,troitsky:2019}.

The object of study in the fundamental paper \cite{aviles_rodriguez_tradacete:2018} is the free Banach lattice over a Banach space, where the applicable notion of freeness is a special case of a general categorical notion. To make this specific, and also as a preparation for the remainder of the paper, we digress briefly to include some material on categories.

The following definition of free objects is that in \cite[Definition~8.22]{adamek_herrlich_strecker_ABSTRACT_AND_CONCRETE_CATEGORIES_THE_JOY_OF_CATS:2006}.

\begin{definition}\label{def:free_object}
	Suppose that $\catone$ and $\cattwo$ are categories, and that $\functor\colon\cattwo\mapsto\catone$ is a faithful functor.\footnote{Recall that $\functor$ is \emph{faithful} when the associated map $\functor\colon\mathrm{Hom}_{\cattwoscript}(O_2,O_2^\prime)\to\mathrm{Hom}_{\catonescript}(\functor(O_2),\functor(O_2^\prime))$ is injective for all objects $O_2,O_2^\prime$ of $\cattwo$}. Take an object $O_1$ of $\catone$. Then a \emph{free object over $O_1$ of $\cattwo$ with respect to $\functor$} is a pair $(j,\FCATONECATTWO{O_1})$, where $\FCATONECATTWO{O_1}$ is an object of $\cattwo$ and $j:O_1\to\FCATONECATTWO{O_1}$ is a morphism of $\catone$,
	with the property that, for every object $O_2$ of $\cattwo$ and every morphism $\varphi:O_1\to O_2$ of $\catone$, there exists a unique morphism  $\fact{\varphi}: \FCATONECATTWO{O_1}\to O_2$ of $\cattwo$ such that the diagram
	\begin{equation*}\label{dia:free_object}
		\begin{tikzcd}
			O_1\arrow[r, "j"]\arrow[dr, "\varphi", swap]& \functor\left(\FCATONECATTWO{O_1}\right)\arrow[d, "\functor(\fact{\varphi})"]
			\\ & \functor(O_2)
		\end{tikzcd}
	\end{equation*}
	in $\catone$ is commutative.
\end{definition}

A pair $(j,\FCATONECATTWO{O_1})$ as in \cref{def:free_object} need not exist. However, if such a pair exists, and if  $(j^\prime,\FCATONECATTWO{O_1}^\prime)$ is another such pair, then a standard argument shows that the unique morphism  $\fact{j^\prime}:\FCATONECATTWO{O_1}\to\FCATONECATTWO{O_1}^\prime$ of $\cattwo$ such that $\functor(\fact{j^\prime})\circ j=j^\prime$ is, in fact, an isomorphism. In particular, a free object over $O_1$ of $\cattwo$ with respect to $\functor$, if it exists, is determined up to an isomorphism of $\cattwo$.  We shall, therefore, sometimes speak of `the' free object $\FCATONECATTWO{O_1}$ over $O_1$ of $\cattwo$, the accompanying morphism $j$ and the functor $\functor$ being tacitly understood from the context. Also, when a pair $(j,O)$ satisfies \cref{def:free_object}, where $O$ is an object of $\cattwo$, then we shall say that $O$ is a \emph{realisation} of the free object over $O_1$ of $\cattwo$, or that $O$ is a realisation of $\FCATONECATTWO{O_1}$.

Free objects have the following transitivity property, which is easily verified. Let $\catone$, $\cattwo$, and $\catthree$ be categories. Suppose that $\functor_{12}:\cattwo\to\catone$ and $\functor_{23}:\catthree\to\cattwo$ are faithful functors, so that also $\functor_{12}\circ \functor_{23}:\catthree\to\catone$ is a faithful functor. Take an object $O_1$ of $\catone$, and suppose that $\FCATONECATTWO{O_1}$ exists in $\cattwo$ with respect to $\functor_{12}$, with accompanying morphism $j_{12}:O_1\to\functor_{12}\left(\FCATONECATTWO{O_1}\right)$ of $\catone$. Suppose that $\FCATTWOCATTHREE{\FCATONECATTWO{O_1}}$ exists in $\catthree$ with respect to $\functor_{23}$, with accompanying morphism $j_{23}:\FCATONECATTWO{O_1}\to\functor_{23}\left(\FCATTWOCATTHREE{\FCATONECATTWO{O_1}}\right)$ of $\cattwo$. Then $\FCATONECATTHREE{O_1}$ exists in $\catthree$ with respect to $\functor_{12}\circ\functor_{23}$. In fact, one can take $\FCATONECATTHREE{O_1}\coloneqq\FCATTWOCATTHREE{\FCATONECATTWO{O_1}}$ and $j_{13}\coloneqq \functor_{12}(j_{23})\circ j_{12}$ as accompanying morphism $j_{13}:O_1\to \left(\functor_{12}\circ\functor_{23}\right)\left(\FCATONECATTHREE{O_1}\right)$ of $\catone$.

In the present paper, we shall be concerned only with the category $\SET$ of sets with arbitrary maps as morphisms; the category $\VS$ of real vector spaces with linear maps as morphisms; the category $\VL$ of real vector lattices with vector lattice homomorphisms as morphisms; the category $\BS$ of real Banach spaces with contractions as morphisms; and the category $\BL$ of real Banach lattices with contractive vector lattice homomorphisms as morphisms. There are canonical forgetful functors between various pairs of these categories, and it is these forgetful functors that will always tacitly be used as the functor $\functor$ in \cref{def:free_object} when speaking of free objects of one of these categories over an object of another of these categories.

We now return to our introductory context of free Banach lattices over Banach spaces. In the categorical terminology of \cref{def:free_object}, it is a part of the statement of \cite[Theorem~2.5]{aviles_rodriguez_tradacete:2018} that $\FBSBL{E}$ exists for every Banach space $E$.\footnote{The definition of a free Banach lattice over a Banach space in \cite[p.~2956]{aviles_rodriguez_tradacete:2018} is easily seen to be equivalent to the one in the present paper with respect to the canonical forgetful functor. Some caution is, however, advised with the terminology of free Banach lattices over sets as it is used in \cite{de_pagter_wickstead:2015} because the definition \cite[Definition~4.1]{de_pagter_wickstead:2015} of such objects involves only bounded (rather than arbitrary) maps from sets into Banach lattices. It is, in fact, not difficult to show that, for an infinite set $S$, $\FSETBL{S}$ in the sense of \cref{def:free_object} does not exist in $\BL$; not even  when one takes arbitrary, not necessarily contractive, vector lattice homomorphisms as the morphisms of $\BL$. Likewise, when $S$ is an infinite set, then $\FSETBS{S}$ in the sense of \cref{def:free_object} does not exist in $\BS$; not even when one takes arbitrary, not necessarily contractive, bounded linear maps as the morphisms of $\BS$.} The full statement is, however, much more precise: $\FBSBL{E}$ has a realisation as an explicitly given Banach lattice $FBL[E]$ of real-valued positively homogeneous functions on the norm dual $E^\ast$ of $E$. The accompanying contraction $j:E\to FBL[E]$ (which is, in fact, an isometric embedding) is obtained by setting
\begin{equation}\label{eq:canonical_embedding}
	[j(x)](x^\ast)\coloneqq x^\ast(x)
\end{equation}
for $x\in E$ and $x^\ast\in E^\ast$. This (ingenious) proof of the existence of $\FBSBL{E}$ has the `drawback' that it is not at all trivial to surmise that the Banach lattice $FBL[E]$ is the sought object. On the other hand: once this has been seen to be true, it can be used fruitfully, in \cite{aviles_rodriguez_tradacete:2018} and its sequels, as the starting point for a further study of the structure of $\FBSBL{E}$ and its relation with $E$.

An alternate, more abstract, approach to the existence of $\FBSBL{E}$ is pointed out in \cite[p.~104--106]{de_jeu:2021}. Given a Banach space $E$, one first considers $E$ merely as a vector space. The free vector lattice $\FVSVL{E}$ over $E$ exists,\footnote{The existence of a free vector lattice over a vector space $V$ is one the statements in \cite[Theorem~6.2]{de_jeu:2021} on the existence of several free objects with a vector lattice structure. Alternatively, one can observe that a free vector lattice over a basis $S$ of $V$ is also a free vector lattice over $V$, and use the classical fact that the free vector lattice over a set exists; see \cite{baker:1968,bleier:1973,birkhoff:1942,weinberg:1963}.} and it is surprisingly easy to introduce a vector lattice seminorm $\rho$ on $\FVSVL{E}$ such that the completion of $\FVSVL{E}/\mathrm{Ker\,}\rho$ is routinely verified to be $\FBSBL{E}$. The contours of this elementary approach, where the norm is introduced via an abstract rather than a more concrete definition as in \cite{aviles_rodriguez_tradacete:2018}, are already visible in \cite{troitsky:2019}.\footnote{The earliest occurrence of this idea that the author is aware of, is the construction of the enveloping $\mathrm{C}^\ast$-algebra of an involutive Banach algebra $A$ with an approximate identity; see \cite[Section~2.7]{dixmier_C-STAR-ALGEBRAS_ENGLISH_NORTH_HOLLAND_EDITION:1977}. Using the fact that every $\mathrm{C}^\ast$-algebra is isomorphic to a $\mathrm{C}^\ast$-algebra of operators on a Hilbert space, it is easily seen that the enveloping $\mathrm{C}^\ast$-algebra of $A$ that is defined in \cite[2.7.2]{dixmier_C-STAR-ALGEBRAS_ENGLISH_NORTH_HOLLAND_EDITION:1977} is precisely the free $\mathrm{C}^\ast$-algebra over $A$ in the sense of \cref{def:free_object} with respect to the canonical forgetful functor.} Reversing the pros and cons of \cite[Theorem~2.5]{aviles_rodriguez_tradacete:2018}, this approach to the existence of $\FBSBL{E}$ has the drawback that it does not yield any information about its structure, but it has the advantage that it is completely elementary, once one has the existence of $\FVSVL{E}$ to start with. Moreover, this approach can also be used to routinely show the existence of Banach lattice algebras with universal properties. Such non-commutative objects cannot be Banach lattice algebras of functions. It is tempting to conjecture that they can be realised as Banach lattice algebras of operators, but it remains to be seen to what extent this is true. At any rate, a proof of their existence `by pointing them out' as in \cite[Theorem~2.5]{aviles_rodriguez_tradacete:2018} seems infeasible at the time of writing.  A method analogous to that on \cite[p.~104--106]{de_jeu:2021} is presently the only way to construct them, where the free vector lattice algebras that are needed as a starting point are supplied by \cite[Theorem~6.2]{de_jeu:2021}.

\medskip

\noindent
The main result of the present paper, \cref{res:free_vector_lattice_over_vector_space_as_lattice_of_functions}, provides a connection between the proofs of the existence of $\FBSBL{E}$ for a Banach space $E$ that are given in \cite{aviles_rodriguez_tradacete:2018} and in \cite{de_jeu:2021}. Its consequence \cref{res:consequences_for_free_vector_lattice_over_vector_space} shows that the approach in \cite{de_jeu:2021} as outlined above can be simplified a little further. We shall now explain this.

\cref{res:free_vector_lattice_over_vector_space_as_lattice_of_functions} asserts that, for a real vector space $V$ and any vector space $L^\sharp$ of linear functionals on $V$ that separates the points of $V$, the free vector lattice $\FVSVL{V}$ over $V$ can be realised as the vector lattice of positively homogeneous functions on $L^\sharp$ that is generated by the functions on $L^\sharp$ that originate from $V$ in the canonical way. In particular, for a Banach space $E$, where $E^\ast$ separates the points of $E$, \cref{res:free_vector_lattice_over_vector_space_as_lattice_of_functions} shows that $\FVSVL{E}$ can be realised as the vector lattice of positively homogeneous functions on $E^\ast$ that is generated by the set $\{j(x): x\in E\}$ of functions on $E^\ast$, where $j(x)$ is defined as in \cref{eq:canonical_embedding}. For each contraction $\varphi:E\to F$, where $F$ is a Banach lattice, we let $\fact{\varphi}: \FVSVL{E}\to F$ denote the unique vector lattice homomorphism such that $\fact{\varphi}\circ j = \varphi$. For $\xi\in\FVSVL{E}$, we then set
\[
\rho(\xi)\coloneqq\sup\{\norm{\fact{\varphi}(\xi)}: F \text{ is a Banach lattice and }\varphi:E\to F \text{ is a contraction}\}.
\]
It is pointed out in \cite{de_jeu:2021} that $\rho$ is a (finite-valued) vector lattice seminorm on $\FVSVL{E}$, and that $\FBSBL{E}$ is the completion of the vector lattice $\FVSVL{E}/ \mathrm{Ker}\,\rho$ in the induced vector lattice norm. What was not yet known at the time of writing of \cite{de_jeu:2021} is that $\rho$ is actually a norm. This follows from \cref{res:consequences_for_free_vector_lattice_over_vector_space}, which implies that the vector lattice homomorphisms $\fact{x^\ast}:\FVSVL{E}\to\RR$ for $x^\ast$ in $E^\ast$ with $\norm{x^\ast}\leq 1$ already separate the points of $\FVSVL{E}$. Consequently, $\FBSBL{E}$ is the completion of the vector lattice $\FVSVL{E}$ of positively homogeneous functions on $E^\ast$ in  the vector lattice norm $\rho$.

The categorical definition of $\FBSBL{E}$ and its equivalent definition in \cite{aviles_rodriguez_tradacete:2018} do not at all point to the possibility of realising it as a Banach lattice of positively homogeneous functions on $E^\ast$. This pleasant surprise is simply `observed' in \cite[Theorem~2.5]{aviles_rodriguez_tradacete:2018}. The fact, obtained by combining \cite[p.~104--106]{de_jeu:2021}, \cref{res:free_vector_lattice_over_vector_space_as_lattice_of_functions}, and \cref{res:consequences_for_free_vector_lattice_over_vector_space}, that $\FBSBL{E}$ is the completion of a vector lattice of positively homogeneous functions on $E^\ast$ does not fully predict this possibility (because a completion of a function space need not be a function space), but it \emph{does} make it more plausible.

The above leads to a natural question. Suppose that $E$ is a Banach space, and that $L$ is a linear subspace of $E^\ast$ that separates the points of $E$. Then an argument completely analogous o the above shows that $\FBSBL{E}$ is the completion of a canonical vector lattice of positively homogeneous functions on $L$ in the norm $\rho$. When is $\FBSBL{E}$ again a Banach lattice of (positively homogeneous) \emph{functions} on $L$, as we know from \cite[Theorem~2.5]{aviles_rodriguez_tradacete:2018} to be the case for $L=E^\ast$?

\medskip

\noindent In addition to its intrinsic interest, and its role in a more intuitive understanding of the nature of $\FBSBL{E}$, our main result \cref{res:free_vector_lattice_over_vector_space_as_lattice_of_functions} has an application to free vector lattices over sets; see \cref{res:minimal_model}.

\section{Free vector lattices over vector spaces as function lattices}\label{sec:free_vector_lattices_over_vector_spaces_as_function_lattices}

\noindent In this section, all vector spaces are over the real numbers. When $V$ is a vector space, then we write $V^\sharp$ for the vector space of all linear functionals on $V$. We let $\co(S)$ denote the convex hull of a non-empty subset $S$ of a vector space. When $S$ is a non-empty set, we let $\Fun(S)$ denote the real-valued functions on $S$, and we write $\Fun_{00}(S)$ for the real-valued functions on $S$ with finite support. We do \emph{not} suppose that vector lattices are Archimedean.

Let $V$ be a vector space. As was observed above, there exists a free vector lattice $(j,\FVSVL{V})$ over $V$. Take a vector space $L^\sharp$ of linear functionals on $V$ that separates the points of $V$. It is our aim to show that $\FVSVL{V}$ is vector lattice isomorphic to a canonical vector lattice of positively homogeneous functions on $L^\sharp$; see \cref{res:free_vector_lattice_over_vector_space_as_lattice_of_functions}. The proofs are inspired by those in \cite{bleier:1973}, where it is shown that, for a non-empty set $S$, the free vector lattice over $S$ can be realised as a function lattice.

Since the linear maps from $V$ into vector lattices separate the points of $V$, the linear map $j:V\to\FVSVL{V}$ is injective. Obviously, the vector sublattice of $\FVSVL{V}$ that is generated by $j(V)$ also has the universal property as in \cref{def:free_object}. It then follows easily from the aforementioned uniqueness of a free object up to isomorphism that $j(V)$ generates $\FVSVL{E}$ as a vector lattice. After identifying $V$ and $j(V)$, we have that $V$ generates $\FVSVL{V}$ as a vector lattice.

We continue with two preparatory lemmas.

\begin{lemma}[Bleier]\label{res:bleier_one}
	Let $E$ be a vector lattice, let $k\geq 1$, and let $e_1,\dotsc,e_k\in E$. Suppose that $0\in\co\{e_1,\dotsc,e_k\}$. Then $\bigwedge_{i=1}^k e_i\leq 0$.
\end{lemma}

\begin{proof}
	Bleier's original proof of \cite[Lemma~2.1]{bleier:1973} uses the non-trivial fact that every vector lattice is a subdirect product of totally ordered vector lattices. For totally ordered vector lattices the result is easily seen to be true, and the general statement then follows. The following alternate proof is completely elementary.
	
	There exist $r_1,\dotsc,r_k\geq 0$ with $\sum_{i=1}^k r_i=1$ such that $\sum_{i=1}^k r_i e_i=0$.
	Take $j$ with $1\leq j\leq k$. Then $\bigwedge_{i=1}^ke_i\leq e_j$, so $r_j\bigwedge_{i=1}^ke_i\leq r_j e_j$. Hence $\bigwedge_{i=1}^k e_i=\sum_{j=1}^k \left(r_j \bigwedge_{i=1}^k e_i\right)\leq\sum_{j=1}^k r_j e_j=0$.
\end{proof}

The following is similar to \cite[Theorem~2.2]{bleier:1973}.

\begin{lemma}\label{res:bleier_two}
	Let $E$ and $F$ be vector lattices, and let $\fact\Psi: E\to F$ be a vector lattice homomorphism. Suppose that $E$ is generated as a vector lattice by a linear subspace $V$ with the property that $0\in\co\{v_1,\dotsc,v_k\}$ whenever $k\geq 1$ and $v_1,\dotsc,v_k\in V$ are such that $\bigwedge_{i=1}^k \fact\Psi(v_i)\leq 0$.
	Then $\fact\Psi$ is injective on $E$.
\end{lemma}

\begin{proof}
	Suppose that $x\in E$ is such that $\fact\Psi(x)=0$. Since the linear subspace $V$ of $E$ generates $E$, there exist an integer $N\geq 1$, integers $M_j\geq 1$ for $j=1,\dotsc,N$, and elements $v_{ji}$ of $V$ for $j=1,\dotsc,N$ and $i=1,\dotsc,M_j$ such that $x=\bigvee_{j=1}^N\bigwedge_{i=1}^{M_j} v_{ji}$; see \cite[Exercise~4.1.8]{aliprantis_burkinshaw_POSITIVE_OPERATORS_SPRINGER_REPRINT:2006}, for example. Then $\bigvee_{j=1}^N\bigwedge_{i=1}^{M_j}\fact\Psi(v_{ji})=0$, so that $\bigwedge_{i=1}^{M_j}\fact\Psi(v_{ji})\leq 0$ for each $j=1,\dotsc,N$. Take such $j$. By hypothesis, we have $0\in\co\{v_{j1},\dotsc,v_{jM_j}\}$. \cref{res:bleier_one} then shows that $\bigwedge_{i=1}^{M_j} v_{ji}\leq 0$. Hence  $x=\bigvee_{j=1}^N\bigwedge_{i=1}^{M_j} v_{ji}\leq 0$. Since also $\fact\Psi(-x)=0$, we likewise conclude that  $-x\leq 0$, so that $x=0$, as required.
\end{proof}

The conclusion of the argument relies on convex analysis in Hausdorff locally convex topological vector spaces.

\begin{theorem}\label{res:free_vector_lattice_over_vector_space_as_lattice_of_functions}
	
	Let $V$ be a vector space, and let $(j,\FVSVL{V})$ be a free vector lattice over $V$. Take a linear subspace $L^\sharp$ of $V^\sharp$ that separates the points of $V$. Define the linear map $\Psi:V\to\Fun(L^\sharp)$ by setting
	\[
	[\Psi(v)](l^\sharp)\coloneqq l^\sharp(v)
	\]
	for $v\in V$ and $l^\sharp\in L^\sharp$. Then the unique extension of $\Psi$ to a vector lattice homomorphism $\fact{\Psi}: \FVSVL{V}\to\Fun(L^\sharp)$ is injective, and establishes a vector lattice isomorphism between $\FVSVL{V}$ and the vector sublattice $\mathcal F$ of $\Fun(L^\sharp)$ that is generated by the linear subspace $\Psi(V)$ of $\Fun(L^\sharp)$. Consequently, $(\Psi,\mathcal F)$ is a free vector lattice over $V$.
\end{theorem}

\begin{proof}
	Since $\FVSVL{V}$ is generated as a vector lattice by its linear subspace $V$, all will be clear once we show that $\fact\Psi$ is injective. For this,  we verify that the hypothesis in \cref{res:bleier_two} is satisfied for the linear subspace $V$ of $\FVSVL{V}$ that generates $\FVSVL{V}$ as a vector lattice. For this, we introduce the weak topology $\sigma(V,L^\sharp)$ on $V$. Since $L^\sharp$ separates the points of $V$, it provides $V$ with the structure of a Hausdorff locally convex topological vector space such that the $\sigma(V,L^\sharp)$-continuous linear functionals on $V$ are precisely the elements of $L^\sharp$; see \cite[Theorem~3.10]{rudin_FUNCTIONAL_ANALYSIS_SECOND_EDITION:1991}, for example.
	
	Suppose that $k\geq 1$ and that $v_1,\dotsc,v_k\in V$ are such that $\bigwedge_{i=1}^k \fact\Psi(v_i)\leq 0$ in $\Fun(L^\sharp)$ or, equivalently, that $\bigwedge_{i=1}^k l^\sharp(v_i)\leq 0$ for all $l^\sharp\in L^\sharp$. We are to show that $0\in\co\{v_1,\dotsc,v_k\}$.
	
	For this, we first note that $\co\{v_1,\dotsc,v_k\}\subseteq\mathrm{span}\,\{v_1,\dotsc,v_k\}$, which is a finite-dimensional linear subspace of $V$. Since there is only one topology on $\mathrm{span}\,\{v_1,\dotsc,v_k\}$ that makes it into a Hausdorff topological vector space (see \cite[Theorem~1.21]{rudin_FUNCTIONAL_ANALYSIS_SECOND_EDITION:1991}, for example), and since the convex hull of finitely many points in a Euclidean space is compact (it is the continuous image of a compact simplex), we see that $\co\{v_1,\dotsc,v_k\}$ is a compact subset of $\mathrm{span}\{v_1,\dotsc,v_k\}$ when this linear subspace of $V$ is supplied with the induced $\sigma(V,L^\sharp)$-topology. Hence $\co\{v_1,\dotsc,v_k\}$ is also a $\sigma(V,L^\sharp)$-compact convex subset of $V$.
	
	A non-empty compact convex subset of a Hausdorff locally convex topological vector space is the intersection of the closed affine half-spaces containing it; see \cite[Corollary~3.11]{conway_A_COURSE_IN_FUNCTIONAL_ANALYSIS_SECOND_EDITION:1990}, for example. Since the $\sigma(V,L^\sharp)$-continuous linear functionals on $V$ are precisely the elements of $L^\sharp$, we see that there exist an index set $A$ and, for each $\alpha\in A$, an element $l_\alpha^\sharp$ of $L^\sharp$ and a real number $r_\alpha$ such that
	\begin{equation}\label{eq:intersection}
		\co\{v_1,\dotsc,v_k\}=\bigcap_{\alpha\in A}\left\{v\in V: l_\alpha^\sharp(v)\geq r_\alpha\right\}.
	\end{equation}
	Take $\alpha_0\in A$. Since $\bigwedge_{i=1}^k l^\sharp(v_i)\leq 0$ for all $l^\sharp\in L^\sharp$, we have $\bigwedge_{i=1}^k l_{\alpha_0}^\sharp(v_i)\leq 0$. Hence we can take an $i_0$ with $1\leq i_0\leq k$ such that $l_{\alpha_0}^\sharp(v_{i_0})\leq 0$. Since $v_{i_0}$ is an element of $\co\{v_1,\dotsc,v_k\}$, it is an element of all closed affine half-spaces occurring in the intersection in the right hand side of \cref{eq:intersection}. In particular, we have that $v_{i_0}\in\left\{v\in V: l_{\alpha_0}^\sharp(v)\geq r_{\alpha_0}\right\}$. The fact that $l_{\alpha_0}^\sharp(v_{i_0})\leq 0$ then implies that $r_{\alpha_0}\leq 0$, so that $0\in\left\{v\in V: l_{\alpha_0}^\sharp(v)\geq r_{\alpha_0}\right\}$. Since $\alpha_0\in A$ was arbitrary, it is now clear from \cref{eq:intersection} that $0\in\co\{v_1,\dotsc,v_k\}$, as desired.
	
	By \cref{res:bleier_two}, $\fact{\Psi}$ is injective. It is then immediate that $\FVSVL{V}$ is Archime\-dean.
\end{proof}

The following is immediate from \cref{res:free_vector_lattice_over_vector_space_as_lattice_of_functions}.

\begin{corollary}\label{res:consequences_for_free_vector_lattice_over_vector_space}
	Let $V$ be a vector space, and let $(j,\FVSVL{V})$ be a free vector lattice over $V$. Then $\FVSVL{V}$ is Archimedean.
	Take a linear subspace $L^\sharp$ of $V^\sharp$ and, for $l^\sharp\in L^\sharp$, let $\fact{l^\sharp}:\FVSVL{V}\to\RR$ be the unique vector lattice homomorphism such that $\fact{l^\sharp}\circ j = l^\sharp$. If $L^\sharp$ separates the points of $V$, then the linear subspace $\{\fact{l^\sharp}: l^\sharp\in L^\sharp\}$ of $\FVSVL{V}^\sharp$ separates the points of $\FVSVL{V}$.
\end{corollary}

We conclude with two comments on \cref{res:free_vector_lattice_over_vector_space_as_lattice_of_functions}.

\begin{remark}\label{rem:minimal_model_for_free_vector_lattie_over_vector_space}
	As already mentioned above, the free vector lattice $\FVSVL{V}$ over a vector space $V$ is the same as the free vector lattice $\FSETVS{S}$ over a basis $S$ of $V$. It is known (see \cite[Theorem~2.4]{baker:1968} or \cite[Theorem~2.3]{bleier:1973}) that $\FSETVS{S}$ can be realised as a vector lattice of functions. Hence the consequence of  \cref{res:free_vector_lattice_over_vector_space_as_lattice_of_functions} that $\FVSVL{V}$ can be realised as a vector lattice of functions is already clear. The added value of \cref{res:free_vector_lattice_over_vector_space_as_lattice_of_functions} lies in the freedom to choose any separating linear subspace $M^\sharp$ of $V^\sharp$ as underlying point set for a realisation of $\FVSVL{V}$ as a vector lattice of functions. To illustrate this, we recall how $\FSETVL{S}$ is classically realised as a vector lattice of functions. For $s\in S$, we define $j: S\to\Fun(\Fun(S))$ by setting $[j(s)](f)\coloneqq f(s)$ for $f\in\Fun(S)$. Let $\mathcal F$ be the vector sublattice of $\Fun(\Fun(S))$ that is generated by $j(S)$. Then $(j,\mathcal F)$ is a free vector lattice over $S$; see \cite[Theorem~2.3]{bleier:1973}. In the case where $S$ is a basis of a vector space $V$, then $\Fun(S)$ can be identified with $V^\sharp$. Using this in the canonical isomorphism between $\FVSVL{V}$ and $\FSETVL{S}$, one sees that the classical result implies \cref{res:free_vector_lattice_over_vector_space_as_lattice_of_functions} for the maximal choice $L^\sharp=V^\sharp$. For our discussion on the free Banach lattice over a Banach space $E$, this consequence of the classical result is, however, not sufficient. When one wants to have an intuition why $\FBSBL{E}$ can be realised as a Banach lattice of functions on $E^\ast$, then the fact that $\FVSVL{E}$ can be realised as a vector lattice of functions on $E^\sharp$ is not sufficient. One needs to be able to choose its separating linear subspace $E^\ast$ in \cref{res:free_vector_lattice_over_vector_space_as_lattice_of_functions} in order to know that  $\FVSVL{E}$ can be realised as a vector lattice of functions on $E^\ast$.
\end{remark}

\begin{remark}\label{rem:application_to_free_vector_lattices_over_sets}
	Let $S$ be a non-empty set. We shall now use \cref{res:free_vector_lattice_over_vector_space_as_lattice_of_functions} to obtain a realisation	of $\FSETVL{S}$ as a vector lattice of functions on an underlying point set that, for infinite $S$, is strictly smaller than the classical one $\Fun(S)$. This is done by exploiting the fact that $\FSETVS{S}=\FVSVL{\FSETVS{S}}$. We start by noting that every vector space with a basis indexed by $S$ is a free vector space over $S$. Consider the vector space $\Fun(S)$ of all real-valued functions on $S$. For each $s\in S$, define $\delta_{s}:S\to\RR$ by setting
	\begin{equation}\label{eq:delta}
		\delta_{s}(s^\prime)\coloneqq
		\begin{cases} 1 & \text{ when } s^\prime=s;\\
			0&\text{ when } s^\prime\neq s.
		\end{cases}
	\end{equation}
	Then the linear span of $\{\delta_s: s\in S\}$ in $\Fun(S)$, i.e., the vector space $\Fun_{00}(S)$, is a free vector space over $S$. We can now apply \cref{res:free_vector_lattice_over_vector_space_as_lattice_of_functions} to realise the free vector lattice $\FVSVL{\Fun_{00}(S)}$ (which coincides with $\FSETVL{S}$) as a vector lattice of positively homogeneous functions on any linear subspace $L^\sharp$ of $\Fun_{00}(S)^\sharp$ that separates the points of $\Fun_{00}(S)$: it is the vector sublattice of $\Fun(L^\sharp)$ that is generated by the functions $l^\sharp\mapsto l^\sharp(f)$ for $f\in\Fun_{00}(S)$. Since the $\delta_s$ span $\Fun_{00}(S)$, it is also the vector sublattice of $\Fun(L^\sharp)$ that is generated by the functions $l^\sharp\mapsto l^\sharp(\delta_s)$ for $s\in S$. What can we take for $L^\sharp$? Obviously, $\Fun_{00}(S)^\sharp$ itself will do. This vector space can be identified with $\Fun(S)$, and upon doing this one sees that \cref{res:free_vector_lattice_over_vector_space_as_lattice_of_functions} for this choice of $L^\sharp$ yields the classical realisation of $\FSETVL{S}$ as a sublattice of $\Fun(\Fun(S))$. There is, however, another canonical choice for a separating linear subspace of  $\Fun_{00}(S)^\sharp$, namely, (an isomorphic image of) $\Fun_{00}(S)$ itself. Indeed, when $f=\sum_{s\in S}\lambda_s\delta_s$ is an expression of an element $f$ of  $\Fun_{00}(S)$ as a finite linear combination of the basis elements $\delta_s$ of $\Fun_{00}(S)$, then one can define $f^\sharp:\Fun_{00}(S)\to\RR$ by setting $f^\sharp(g)\coloneqq\sum_{s,s^\prime\in S}\lambda_s\lambda_{s^\prime}$ for $g=\sum_{s^\prime\in S}\lambda_{s^\prime}\delta_{s^\prime}\in\Fun_{00}(S)$. For this choice of $L^\sharp$, \cref{res:free_vector_lattice_over_vector_space_as_lattice_of_functions} yields the following.
	It shows that $\FSETVL{S}$ can also be realised as a vector lattice of functions on $\Fun_{00}(S)$ which, for infinite $S$, is a proper subset of $\Fun(S)$.

	\begin{proposition}\label{res:minimal_model}
		Let $S$ be a non-empty set. For $s\in S$, define $j(s):\Fun_{00}(S)\to\RR$ by setting $j(s)\coloneqq f(s)$ for $f\in\Fun_{00}(S)$. Let $\mathcal F$ be the vector sublattice of $\Fun(\Fun_{00}(S))$ that is generated by $j(S)$. Then $(j,\mathcal F)$ is a free vector lattice over $S$.	
	\end{proposition}
	
	The above approach to $\FSETVS{S}$ via \cref{res:free_vector_lattice_over_vector_space_as_lattice_of_functions} and the relation $\FSETVS{S}=\FVSVL{\FSETVS{S}}$ may also prompt one to think differently of the role of $\Fun_{00}(S)$ in \cref{res:minimal_model} and of $\Fun(S)$ in the classical result. Conceptually, they are not merely underlying point sets for realisations of $\FSETVL{S}$ as vector lattices of functions, as one might be inclined to think at first sight. They should be viewed as separating linear subspaces of $\Fun_{00}(S)^\sharp$, and $\FSETVL{S}$ can be realised as a vector lattice of positively homogeneous functions on each of these two vector spaces. Furthermore, it is a consequence of \cref{res:minimal_model} that the free vector lattice over a non-empty set $S$ can be realised as a vector lattice of positively homogeneous functions on the free vector space over $S$.
\end{remark}

\subsection*{Acknowledgement} The author is grateful to the anonymous referee for their careful reading of the manuscript.

\bibliographystyle{plain}
\urlstyle{same}

\bibliography{general_bibliography}

\end{document}